\documentclass[12pt,oneside]{amsart}
\usepackage{amssymb,amsmath,latexsym,epsfig}

\usepackage{euler, amsfonts, amssymb, latexsym, epic}

\numberwithin{equation}{section}
\begin{document}


\newcommand{\az}{\alpha_{Z}}
\newcommand{\bz}{\beta_{Z}}
\newcommand{\ax}{\alpha_{\X}}\newcommand{\bx}{\beta_{\X}}
\newcommand{\popo}{\mathbb{P}^1 \times \mathbb{P}^1}
\newcommand{\pnpm}{\mathbb{P}^n \times \mathbb{P}^m}
\newcommand{\pnr}{\mathbb{P}^{n_1}\times \cdots \times \mathbb{P}^{n_r}}
\newcommand{\pnk}{\mathbb{P}^{n_1}\times \cdots \times \mathbb{P}^{n_k}}
\newcommand{\prthree}{\mathbb{P}^1 \times \mathbb{P}^1 \times \mathbb{P}^1}
\newcommand{\Iz}{I_{Z}}
\newcommand{\Ix}{I_{\X}}
\newcommand{\C}{\mathcal{C}}
\newcommand{\Y}{\mathbb{Y}}
\newcommand{\Z}{\mathbb{Z}}
\newcommand{\Zr}{\mathbb{Z}_{red}}
\newcommand{\N}{\mathbb{N}}
\newcommand{\pr}{\mathbb{P}}
\newcommand{\X}{\mathbb{X}}
\newcommand{\F}{\mathbb{F}}

\newcommand{\supp}{\operatorname{Supp}}
\newcommand{\depth}{\operatorname{depth}}
\newcommand{\ol}{\overline{L}}
\newcommand{\Ssx}{\mathcal S_{\X}}
\newcommand{\Ss}{\mathcal S}
\newcommand{\dtz}{\Delta H_{Z}}
\newcommand{\dtc}{\Delta^{C} H_{Z}}
\newcommand{\dtcc}{\Delta^{C} H_{Z_{ij}}}
\newcommand{\dt}{\Delta}
\newcommand{\B}{\mathcal{B}}
\newcommand{\ay}{\alpha_{Y}}
\newcommand{\by}{\beta_{Y}}
\newcommand{\ui}{\underline{i}}
\newcommand{\ua}{\underline{\alpha}}
\newcommand{\uj}{\underline{j}}

\newtheorem{theorem}{Theorem}[section]
\newtheorem{corollary}[theorem]{Corollary}
\newtheorem{proposition}[theorem]{Proposition}
\newtheorem{lemma}[theorem]{Lemma}
\newtheorem{alg}{Algorithm}
\newtheorem{question}{Question}
\newtheorem{problem}{Problem}
\newtheorem{conjecture}[theorem]{Conjecture}

\theoremstyle{definition}
\newtheorem{definition}[theorem]{Definition}
\newtheorem{remark}[theorem]{Remark}
\newtheorem{example}[theorem]{Example}
\newtheorem{acknowledgement}{Acknowledgement}
\newtheorem{notation}[theorem]{Notation}
\newtheorem{construction}[theorem]{Construction}


\title{Separators of points in a multiprojective space}
\thanks{Version: July 5, 2007 (Final)}
\author{Elena Guardo}
\address{Dipartimento di Matematica e Informatica\\
Viale A. Doria, 6 - 95100 - Catania, Italy}
\email{guardo@dmi.unict.it}

\author{Adam Van Tuyl}
\address{Department of Mathematics \\
Lakehead University \\
Thunder Bay, ON P7B 5E1, Canada}
\email{avantuyl@sleet.lakeheadu.ca}


\begin{abstract}
In this note we develop some of the properties of
separators of points in a multiprojective
space. In particular, we prove multigraded
analogs of results of Geramita, Maroscia, and Roberts
relating the Hilbert function of $\X$ and $\X \setminus \{P\}$
via the degree of a separator, and Abrescia, Bazzotti, and Marino 
relating the degree of a separator to shifts in
the minimal multigraded free resolution of the ideal of points.
\end{abstract}
\maketitle


\section{Introduction}

Let
$R = k[x_{1,0},\ldots,x_{1,n_1},\ldots,x_{r,0},\ldots,x_{r,n_r}]$
be the $\N^r$-graded polynomial ring with  $\deg x_{i,j} = e_i$, the $i$th standard basis vector in
$\N^r$, and $k$ an algebraically closed field of characteristic zero.
If $\X = \{P_1,\ldots,P_s\}$ is a finite set of points in
a multiprojective space $\pnr$, then $R/\Ix$ is
the associated $\N^r$-graded coordinate ring.   If $P \in \X$,
then the multihomogeneous form $F \in R$ is a {\bf separator for $P$} if
$F(P) \neq 0$ and $F(Q) = 0$ for all $Q \in \X \setminus \{P\}$.  The
{\bf degree of a point $P \in \X$} is the set
\[\deg_{\X}(P) = \min\{\deg F ~|~ F ~~\mbox{is a separator for $P \in \X$}\}.\]
Here, we are using the partial order on $\N^r$ defined by
$(i_1,\ldots,i_r) \succeq (j_1,\ldots,j_r)$ whenever $i_t \geq j_t$ for
all $t= 1,\ldots,r$.
The goal of this note is to record some of the properties of a separator
of a point and its degree in a multigraded context.

The notion of a separator was first introduced for sets of points $\X$ in $\pr^n$
by Orecchia \cite{O} to investigate the conductor of $A = R/\Ix$,
that is, the largest ideal $J$ of $A$ that corresponds
with its extension in the integral closure $\overline{A}$.
It was shown that
the degrees of the minimal generators of $J$ corresponded to the degrees of
the points $P \in \X$.  As later shown by Geramita, Maroscia, and Roberts \cite{gmr},
the degree of a point $P$ allows one to relate the Hilbert function
of $\X$ to that of $\X\setminus \{P\}$.  Abrescia, Bazzotti, and Marino \cite{abm}
demonstrated that $\deg_{\X}(P)$ was also linked to
the shifts appearing in the minimal free graded resolution of $R/\Ix$.  Further
properties of separators in the graded case can be found in
\cite{b,b2,guida,HO,S}, among others.

The study of separators of points in a multigraded setting was initiated by Marino \cite{m1,m2,m3}
who studied separators of points in $\popo$.  Note that when $r \geq 2$,
then it may happen that $|\deg_{\X}(P)| \geq 2$, thus
presenting one of the fundamental differences between
the study of separators of points in $\pr^n$ versus those in $\pnr$.   Marino showed
that
$\X \subseteq \popo$ is arithmetically Cohen-Macaulay (ACM) if and only
if for every $P \in \X$, $|\deg_{\X}(P)| = 1$.  More recently,
the authors \cite{GVT1} extended some
of Marino's results to an arbitrary multiprojective space; in particular,
if $\X \subseteq \pnr$ and is ACM, then every point $P \in \X$ has $|\deg_{\X}(P)| = 1$,
but the converse no longer holds.

While a cursory introduction to the properties of separators appears
in \cite{GVT1}, in this paper we wish to provide a more systematic
introduction,
thereby extending our understanding of
points in a multiprojective space (see, for example,
\cite{GuMaRa,GuMaRa2,GuMaRa3,Gu2,GVT,SVT,VT1,VT2}, for
more on these points).   In Section 2,
we relate the Hilbert functions of
$\X$ and $\X \setminus \{P\}$ using the set $\deg_{\X}(P)$ (see
Theorem \ref{boundsonHx}), thus
introducing a multigraded analog of a result of Geramita, Maroscia, and Roberts \cite{gmr}.
The main result (Theorem \ref{pnrsepfrombetti}) of Section 3 relates $\deg_{\X}(P)$
to the shifts at the end of the multigraded resolution of $R/\Ix$ when
$\X$ is ACM.  This result extends a result of Abrescia, Bazzotti,
and Marino \cite{abm} first proved for separators of points in $\pr^n$.
In the final section, we restrict to the case of ACM points in $\popo$
and their separators.  In particular, we show  (see Theorems \ref{acm} and \ref{nonacm})
that the converse of Theorem \ref{pnrsepfrombetti} holds in $\popo$.

\noindent
{\bf Acknowledgements.}  We would like to thank T\`ai Huy H\`a for his comments on
an earlier draft of this paper.  The second author acknowledges the support of NSERC.
The computer program {\tt CoCoA} \cite{Co} was used during the preliminary 
stages of this paper.


\section{Separators, Hilbert functions, and ACMness}

We continue to use the notation from the introduction.  If
$S \subseteq \N^r$, then $\min S$ denotes the set of the
minimal elements of $S$ with respect to the partial ordering $\succeq$. For
any $\ui\in \N^r$, define $D_{\ui}:=\{\uj \in \N^r ~|~ \uj \succeq
\ui\}.$ For any finite set $S=\{\underline{s}_1,\ldots,\underline{s}_{p}\}\subseteq \N^r$,
we set
\[D_S:=\bigcup_{\underline{s}\in S} D_{\underline{s}}.\]
Note that $ \min D_S = S$; thus $D_S$ can be viewed as the largest
subset of $\N^r$ whose minimal elements are the elements of $S$.

Let $\X$ be a set of distinct points in $\pnr$ and $P \in \X$. We say that the
multihomogeneous form $F\in R$ is
a {\bf minimal separator for} $P$ if $F$ is a separator for $P$,
and if there does not exist a separator $G$ for $P$
with $\deg G \prec \deg F$.  Note that
\[
\deg_{\X}(P)
=\{\deg F ~|~ \mbox{$F$ is a minimal separator of $P \in \X$}\}.
\]
\begin{lemma}\label{existseparator}
Let $\X \subseteq \pnr$ be a set of points and let $P \in \X$.
Then for every $\ui \in D_{\deg_{\X}(P)}$ there exists a form $F$
with $\deg F = \ui$ that is a separator of $P$.
\end{lemma}

\begin{proof}  Fix a $P \in \X$.
For each $i = 1,\ldots,r$, there exists a form $L_i$ with $\deg
L_i =e_i$ such that $L_i(P) \neq 0$.
Geometrically picking $L_i$ corresponds to picking a hyperplane in
$\pr^{n_i}$ that misses the $i$th coordinates of the points of
$\X$.  For any $\ui \in D_{\deg_{\X}(P)}$, there exists $\ua \in
\deg_{\X}(P)$ with $\ui \succeq \ua$. Let $F'$ be a minimal
separator of $P$ with $\deg F' = \ua$.  Then the desired separator
is
\[F = F'\prod_{j=1}^r L_j^{i_j - \alpha_j}\]
where $\ui = (i_1,\ldots,i_r)$ and $\ua =
(\alpha_1,\ldots,\alpha_r)$.
\end{proof}

If $I$ is a multihomogeneous ideal
of $R$, then the {\bf Hilbert function of $S = R/I$ } is the numerical
function $H_{S}:\N^r \rightarrow \N$ defined by \[H_S(\ui):=
\dim_{k} S_{\ui} = \dim_{k} R_{\ui} - \dim_{k}(I)_{\ui}~~\mbox{for
all $\ui \in \N^r$.}\] When $S = R/I_{\X}$ is the coordinate ring
of a set of points $\X$, then we usually say $H_S$ is the {\bf
Hilbert function of } $\X$, and write $H_{\X}$.

If $P \in \X$
and $\Y = \X \setminus \{P\}$, then $H_{\Y}$ can be computed
from $H_{\X}$ and $\deg_{\X}(P)$ as demonstrated below.  
We view this result as a multigraded version of \cite[Lemma 2.3]{gmr}.

\begin{theorem} \label{boundsonHx}
Let $\X$ be a set of distinct points in $\pnr$, and let $P \in \X$
be any point.  If $\Y = \X \backslash \{P\}$, then \[H_{\Y}(\ui) =
\begin{cases} H_{\X}(\ui) & \text{if $\ui \notin D_{\deg_{\X}(P)}$} \\
H_{\X}(\ui)-1   & \text{if $\ui \in D_{\deg_{\X}(P)}$.}
\end{cases}\]
\end{theorem}

\begin{proof}
It was shown in \cite[Theorem 5.3]{GVT1} that there exists a finite
set $S \subseteq \N^r$ such that
\[H_{\Y}(\ui) =
\begin{cases} H_{\X}(\ui) & \text{if $\ui \notin D_{S}$} \\
H_{\X}(\ui)-1   & \text{if $\ui \in D_{S}$.}
\end{cases}\]
It therefore  suffices to show that
$S = \deg_{\X}(P)$. Suppose $\ui \not\in D_{\deg_{\X}(P)}$ but
$H_{\Y}(\ui) = H_{\X}(\ui)-1$. This implies that $\dim_k
(I_{\Y})_{\ui} = \dim_k (I_{\X})_{\ui}+1$, or equivalently, there
exists a form $F \in (I_{\Y})_{\ui}\backslash (I_{\X})_{\ui}$.
But then $F$ vanishes at all the points of $\Y$ but not at all
the points of $\X$, i.e., $F$ does not vanish at $P$.  So $F$ is
a separator of $P$, and thus there exists an $\ua \in
\deg_{\X}(P)$ such that $\ui \succeq \ua$.  But this contradicts
the fact that $\ui \not\in D_{\deg_{\X}(P)}$.  So $H_{\Y}(\ui) =
H_{\X}(\ui)$.

Now suppose that $\ui \in D_{\deg_{\X}(P)}$ but $H_{\Y}(\ui) =
H_{\X}(\ui)$. By Lemma \ref{existseparator} there exists a form $F$
with $\deg F = \ui$ such that $F$ is a separator of $P$.  So $F
\in (I_{\Y})_{\ui}$ but $F \not\in (I_{\X})_{\ui}$.   This
contradicts the fact that  $H_{\Y}(\ui) = H_{\X}(\ui)$ implies
$\dim_k (I_{\Y})_{\ui}= \dim_k (I_{\X})_{\ui}$.  So
 $H_{\Y}(\ui) = H_{\X}(\ui) -1$.
\end{proof}

\begin{remark} \label{hxremark}
Theorem \ref{boundsonHx} shows that
$\deg_{\X}(P) = \min\{\ui \in \N^r ~|~ H_{\X}(\ui) \neq
H_{\Y}(\ui)\}$.  One can therefore compute $\deg_{\X}(P)$
by comparing the Hilbert functions of $\X$ and $\Y = \X \setminus \{P\}$.
\end{remark}

If $\X \subseteq \pr^n$, then $\N$ is a totally ordered set, so we
can study {\it the} degree of a point $P \in \X$ (as in
\cite{abm,b,b2,guida,HO,O,S}). In the
multigraded case the set $\deg_{\X}(P) =
\{\ua_1,\ldots,\ua_s\} \subseteq \N^r$ may have $s \geq 1$.  
However, if $F$ is a minimal separator of $P$ with
$\deg F = \ua_i \in \deg_{\X}(P)$, then the equivalence
class of $F$ in $R/\Ix$ is unique
up to scalar multiplication.

\begin{theorem}[{\cite[Corollary 5.4]{GVT1}}]\label{degreeunique}
Suppose $\deg_{\X}(P) = \{\ua_1,\ldots,\ua_s\} \subseteq \N^r$.
If $F$ and $G$ are any two minimal separators of $P$ with
$\deg F = \deg G = \ua_i$, then there exists $0 \neq c \in k$
such that $\overline{G} = \overline{cF} \in R/\Ix$.
\end{theorem}

If $\Y = \X \setminus \{P\}$ for some $P \in \X$, then the defining
ideals of $I_P$, $I_\Y$ and $I_{\X}$ are related via the
separators of $P$, as demonstrated below.

\begin{theorem} \label{teopt} Let $\X$ be a set of points in $\pnr$,
$P \in \X$, and $\Y = \X \setminus \{P\}$.
\begin{enumerate}
\item[$(i)$]
If $F$ is a separator of a point $P$, then $(I_{\X}:F) = I_P$.
\item[$(ii)$]  If $\deg_{\X}(P) = \{\ua_1,\ldots,\ua_s\}$, and
if $F_i$ is a minimal separator of $P$ with
$\deg F_i = \ua_i$,  then
$I_{\Y} = (I_{\X},F_1,\ldots,F_s)$.
\end{enumerate}
\end{theorem}

\begin{proof}
Statement $(i)$ is \cite[Theorem 5.5]{GVT1}.  For $(ii)$,
the containment $(I_{\X},F_1,\ldots,F_s) \subseteq I_{\Y}$ is
clear since $F_i \in I_{\Y}$ for each $i$ and $I_{\X} \subseteq
I_{\Y}$.
Now, if $\ui \not\in D_{\deg_{\X}(P)}$, then
$(I_{\X},F_1,\ldots,F_s)_{\ui} = (I_{\X})_{\ui} = (I_{\Y})_{\ui}$
where the last equality is a consequence of Theorem
\ref{boundsonHx}.  On the other hand, if $\ui \in
D_{\deg_{\X}(P)}$, then
\[\dim_k (I_{\X})_{\ui} < \dim_k (I_{\X},F_1,\ldots,F_s)_{\ui} \leq
\dim_k (I_{\Y})_{\ui} \leq \dim_k (I_{\X})_{\ui}  + 1.\]
The last inequality follows from Theorem
\ref{boundsonHx}.  We are
forced to have $\dim_k (I_{\X})_{\ui} + 1 = \dim_k
(I_{\X},F_1,\ldots,F_s)_{\ui} = \dim_k (I_{\Y})_{\ui}$, i.e.,
$(I_{\X},F_1,\ldots,F_s)_{\ui} = (I_{\Y})_{\ui}$. Since
$(I_{\X},F_1,\ldots,F_s) \subseteq I_{\Y}$, and
$(I_{\X},F_1,\ldots,F_s)_{\ui} = (I_{\Y})_{\ui}$ for all $\ui \in
\N^r$, this completes the proof.
\end{proof}

We end this section discussing the connection
between separators and the ACMness of a set of points.
For any finite set of points $\X \subseteq \pnr$, it
can be shown (see, for example \cite[Theorem 2.1]{GVT1})
that $\dim R/I_{\X} = r$  and $1 \leq \depth R/I_{\X} \leq r.$
When $\depth R/I_\X = r$, then we say $\X$ is
{\bf arithmetically Cohen-Macaulay} (ACM).
Although it remains an open problem to classify
ACM sets of points in a multiprojective space (see \cite{GVT1} for
some work on this problem), it can be shown that
the separators of ACM sets of points have a particularly
nice property:

\begin{theorem}[{\cite[Theorem 5.7]{GVT1}}]\label{uniqueMultsep}
Let $\X$ be any ACM set of points in $\pnr$.   Then $|\deg_{\X}(P)|  = 1$
for every $P \in \X$.
\end{theorem}

In the case of ACM sets of points in $\pnr$, we can talk about
{\it the} degree of a point, and in this case we usually abuse
notation and write $\deg_{\X}(P) = \alpha$ instead of $\deg_{\X}(P) =\{\alpha\}$.
Although the converse of Theorem \ref{uniqueMultsep} fails to
hold in general (see \cite[Example 5.10]{GVT1}), the converse holds in $\popo$ as first demonstrated
by Marino:

\begin{theorem}[\cite{m3}]\label{marino}
Let $\X$ be a finite set of points in $\popo$.  Then
$\X$ is ACM if and only if
$|\deg_{\X}(P)| = 1$ for every $P \in \X$.
\end{theorem}


\section{The degree of a point and the minimal resolution}

By Theorem \ref{boundsonHx}, if we can compute
$\deg_{\X}(P)$, then $H_{\Y}$ can be computed from $H_{\X}$
where $\Y = \X \setminus \{P\}$.  It is
therefore of interest to identify what finite subsets $S
\subseteq \N^r$ can be the degree of a point.  In this
section, we show that under the extra
hypothesis that $\X$ is ACM,  information about $\deg_{\X}(P)$ can
be read from the last shift in the minimal multigraded resolution
of $\Ix$.  Our result can
be seen as a multigraded analog of a theorem of Abrescia,
Bazzotti, and Marino \cite{abm}.  We begin with a lemma.

\begin{lemma} \label{residealpt}
Let $P \in \pnr$ be any point.  Then the minimal $\N^r$-graded
free resolution of $R/I_{P}$ has the form
\[0 \rightarrow \mathbb{G}_{t} \rightarrow \mathbb{G}_{t-1}
\rightarrow \cdots \rightarrow \mathbb{G}_1 \rightarrow R
\rightarrow
 R/I_{P} \rightarrow 0\]
where $t = \sum_{i=1}^r n_i$ and $\mathbb{G}_t =
R(-n_1,-n_2,\ldots,-n_r)$.
\end{lemma}

\begin{proof}
Because $I_P$ is a complete intersection,
the conclusions follow from the Koszul resolution,
taking into account the multigrading.
\end{proof}

\begin{theorem}\label{pnrsepfrombetti}
Let $\X$ be a finite set of  points in $\pnr$, and furthermore,
suppose that $\X$ is ACM.  Let $P \in \X$, and suppose
that $\deg_{\X}(P) = \underline{\alpha} = (\alpha_1,\ldots,\alpha_r)$.  Let
\[ 0 \rightarrow \F_t =\bigoplus_{\ui \in S_t} R(-\ui)
{\rightarrow}\cdots\rightarrow \F_1 \rightarrow R \rightarrow
R/I_{\X} \rightarrow 0\] be the minimal $\N^r$-graded free
resolution of $R/I_{\X}$ where $t = \sum_{i=1}^r n_i$.
\begin{enumerate}
\item[$(i)$]  If $\Y = \X \setminus \{P\}$ is ACM, then
$(\alpha_1+n_1,\ldots,\alpha_r+n_r) \in S_t$, that is,
$\deg_{\X}(P)+(n_1,\ldots,n_r)$ appears as a shift in the last
free $R$-module.
\item[$(ii)$] If $\Y = \X \setminus \{P\}$ is not ACM, then
$\depth(R/I_{\Y}) = r-1$.
\end{enumerate}
\end{theorem}

\begin{proof}
Because $\X$ is ACM,
by Theorem \ref{uniqueMultsep} $\deg_{\X}(P) = \alpha$
for some $\alpha \in \N^r$.  Let
$F$ be any minimal separator of $P$.  Hence $\deg F = \alpha$,
and by Theorem \ref{teopt},
$I_{P}=(I_{\X}:F)$ and $(I_{\X},F)=I_{\Y}$.   We then have the
short exact sequence
\begin{equation}\label{ses}
0\rightarrow  R/(I_{\X}:F) (-\alpha) = R/ I_{P} (-\alpha)
\stackrel{\times \overline{F}}{\rightarrow} R/I_{\X} \rightarrow
R/(I_{\X},F) = R/I_{\Y} \rightarrow 0.
\end{equation}
By Lemma \ref{residealpt}, the resolution of $R/I_P$ has form
\[
0 \rightarrow R(-n_1,\ldots,-n_r) \rightarrow \mathbb{G}_{t-1}
\rightarrow \cdots \rightarrow \mathbb{G}_1 \rightarrow R
\rightarrow
 R/I_{P} \rightarrow 0.\]
Applying the mapping cone construction to
(\ref{ses}) we get a resolution of $R/I_{\Y}$:
\begin{equation}\label{multires}
\mathcal{H}:~ 0 \rightarrow R(-\alpha_1-n_1,\ldots,-\alpha_r-n_r)
\rightarrow \mathbb{F}_t \oplus
\mathbb{G}_{t-1}(-\underline{\alpha}) \rightarrow \cdots
\end{equation}
\[ \rightarrow \mathbb{F}_2 \oplus \mathbb{G}_{1}(-\underline{\alpha})
\rightarrow \mathbb{F}_1 \oplus R(-\underline{\alpha}) \rightarrow
R \rightarrow R/I_{\Y} \rightarrow 0.\]

If $\Y$ is ACM, then the above resolution cannot be minimal because
it is too long.  So $\mathcal{H} = \mathcal{F} \oplus \mathcal{G}$
where $\mathcal{F}$ is the minimal resolution of $R/I_{\Y}$ and
$\mathcal{G}$ is isomorphic to the trivial complex (see
\cite[Theorem 20.2]{E}).
In particular $R(-\alpha_1-n_1,\ldots,-\alpha_r-n_2)$ must be part
of the trivial complex $\mathcal{G}$, and thus, to obtain a minimal resolution,
$R(-\alpha_1-n_1,\ldots,-\alpha_r-n_r)$ must cancel with
something in  $\mathbb{F}_t \oplus
\mathbb{G}_{t-1}(-\underline{\alpha})$.  Since
$R(-\alpha_1-n_1,\ldots,-\alpha_r-n_r)$ does not appear in
$\mathbb{G}_{t-1}(-\underline{\alpha})$, there exists a shift
$\ui \in S_t$ such that $\ui =
(\alpha_1+n_1,\ldots,\alpha_r+n_r$), thus proving $(i)$.

For $(ii)$ the mapping cone resolution gives a (not
necessarily minimal) resolution of $R/I_{\Y}$ that cannot be
shortened, because otherwise $\Y$ would be ACM.  So, $R/I_{\Y}$ has projective dimension $t+1$.
Now apply the Auslander-Buchsbaum formula.
\end{proof}

We can use the above result to show that some sets of points are not ACM.

\begin{corollary}
Let $\X$ be an ACM scheme in $\pnr$, and $P \in\X$.  If $\deg_{\X}(P) +
(n_1,\ldots,n_r)$ is not a shift of the last syzygy module of $R/\Ix$,
then $\Y = \X \setminus \{P\}$ is not ACM.
\end{corollary}

\begin{example}
The converse of Theorem \ref{pnrsepfrombetti} (i) is false in general
(we will show it is true in $\popo$ in the next section). Let
$P_1,\ldots,P_6$ be six points in general position (that is,
no more than two points on a line, and no five points on a conic)
in $\pr^2$,
and set $Q_{i,j} := P_i \times P_j \in \pr^2 \times \pr^2$.
Consider the following set of 28 points:
\begin{eqnarray*}
\X &=&
\{Q_{1,1},Q_{1,2},Q_{1,3},Q_{1,4},Q_{1,5},Q_{1,6},Q_{2,1},Q_{2,2},Q_{2,3},Q_{2,4},
Q_{2,6}, Q_{3,1},Q_{3,2},Q_{3,5},
Q_{3,6}, \\
&&Q_{4,1},Q_{4,2},Q_{4,5},Q_{4,6},
Q_{5,1},Q_{5,3},Q_{5,6},Q_{6,1},Q_{6,2},Q_{6,3},
Q_{6,4},Q_{6,5},Q_{6,6}\}.
\end{eqnarray*}
Then $\X$ is ACM, since the minimal bigraded resolution has form
\[0 \rightarrow R(-3,-4)^4 \oplus R(-4,-3)^4 \oplus R(-4,-4) \rightarrow
R^{32} \rightarrow R^{38} \rightarrow R^{16} \rightarrow R
\rightarrow R/\Ix \rightarrow 0\] where we have suppressed all
the other bigraded shifts.

We remove the point $Q_{2,2}$ to form the set $\Y = \X \setminus
\{Q_{2,2}\}$. By comparing the Hilbert functions of $\Y$ and
$\X$ (see Remark \ref{hxremark}), we find that $\deg_{\X}(Q_{2,2}) = (2,2)$.  Now
$\deg_{\X}(Q_{2,2}) + (2,2) = (4,4)$ is a shift that appears in
the minimal multigraded resolution of $\Ix$.  However, $\Y$ is
not ACM because $\Y$ is the nonACM set of points of \cite[Example 3.3]{GVT1}.
\end{example}

\begin{example}
In the proof of Theorem \ref{pnrsepfrombetti}, we saw that
$(\ref{multires})$ gives a resolution of $I_{\Y}$. When $\Y = \X
\setminus \{P\}$ is ACM, this resolution is not minimal because
the resolution can be shortened. However, even when we shorten
the resolution by cancelling out
$R(-\alpha-n_1,\ldots,-\alpha_r-n_r)$ with a term in
$\mathbb{F}_{t-1}$, the resulting resolution may still not be
minimal.

For example, let $P_{i,j} := [1:i] \times [1:j] \in \popo$, and
consider the set
\footnotesize
\[\X  =  \{P_{1,1},P_{1,2},P_{1,3},P_{1,4},P_{1,5}, P_{2,1},P_{2,2},P_{2,3},P_{2,4},P_{3,1},
P_{3,2},P_{3,3},P_{3,4},P_{4,1},P_{4,2},P_{4,3}, P_{5,1},P_{5,2}, P_{6,1}\}.
\]
\normalsize
Set $\Y=\X \setminus \{P_{3,4}\}$. We have that
$\deg_{\X}(P_{3,4})=(2,3)$;  in fact, a minimal separator is $F =
L_1L_2R_1R_2R_3$ where $L_i = ix_0 - x_1$ is the degree $(1,0)$ form that
passes through $[1:i]$ and $R_j = jy_0-y_1 $ is the degree $(0,1)$ form that
passes through $[1:j]$ in $R = k[x_0,x_1,y_0,y_1]$. The resolution of $R/\Ix$ is:
\footnotesize
\[0\rightarrow R(-1,-5)\oplus R(-3,-4)\oplus R(-4,-3)\oplus R(-5,-2)\oplus R(-6,-1)
\rightarrow \] \[\rightarrow R(0,-5)\oplus R(-1,-4)\oplus
R(-3,-3)\oplus R(-4,-2)\oplus R(-5,-1)\oplus R(-6,0)\rightarrow R
\rightarrow R/I_{\X}\rightarrow 0\,.\]
\normalsize
The mapping cone
construction gives the resolution:
\footnotesize
\[0 \rightarrow R(-3,-4)\rightarrow R(-3,-4)\oplus R(-1,-5)\oplus
R(-4,-3)\oplus R(-5,-2)\oplus R(-6,-1)\oplus R(-2,-4)\oplus
R(-3,-3)\]
\[ \rightarrow R(0,-5)\oplus R(-1,-4) \oplus R(-4,-2)\oplus
R(-5,-1)\oplus R(-6,0)\oplus R(-3,-3)\oplus
R(-2,-3)\]\[\rightarrow R \rightarrow R/I_{\Y}\rightarrow 0.\]
\normalsize
Since $\Y$ is ACM, the terms $R(-3,-4)$ at the last and second
last step cancel out.  However, the remaining resolution is not
a minimal resolution because the minimal resolution of $\Y$ is
\footnotesize
\[0\rightarrow R(-1,-5)\oplus  R(-4,-3)\oplus R(-5,-2)\oplus
R(-6,-1)\oplus R(-2,-4) \rightarrow \]\[R(0,-5)\oplus R(-1,-4)
\oplus R(-4,-2)\oplus R(-5,-1)\oplus R(-6,0)\oplus
R(-2,-3)\rightarrow R \rightarrow R/I_{\Y}\rightarrow 0.\]
\normalsize
 The
resolution is not minimal in this case because although $I_{\Y} =
(\Ix,F)$,  one of the minimal generators of $\Ix$ is actually a
multiple of $F$.  Precisely, $G = L_1L_2L_3R_1R_2R_3$ is a
minimal generator of degree $(3,3)$ in $\Ix$, and clearly $F =
L_3G$.
\end{example}


\section{Separators of ACM points in $\popo$ and resolutions}

When $\X$ is a set of ACM points in $\popo$, we can improve upon
the results of the last section.  We will show
that the converse of
Theorem \ref{pnrsepfrombetti} $(i)$ holds for points in $\popo$.
Furthermore, we demonstrate that when  $\X$ is ACM, but $\Y = \X \setminus \{P\}$ is not ACM,
then the mapping cone construction used in proof of Theorem \ref{pnrsepfrombetti}
gives a minimal resolution of $I_{\Y}$.
In order to prove these results,
we make use of properties of ACM sets of points in $\popo$
developed in \cite{GVT1,VT1,VT2};  we begin with
a review of this material.

\subsection{ACM sets of points in $\popo$}

If $\X \subseteq \popo$ is a finite set of points, let $\pi_1(\X) = \{P_1,\ldots,P_r\}$,
respectively, $\pi_2(\X) =\{Q_1,\ldots,Q_t\}$, denote the distinct first coordinates,
respectively, second coordinates, of the points $\X$.  Each
point in $\X$ therefore can be written as $P_i \times Q_j$ for some $i$ and $j$;
the corresponding defining ideal is then $I_{P_i \times Q_j} = (L_{P_i},L_{Q_j}) \subseteq
R = k[x_0,x_1,y_0,y_1]$ with $\deg L_{P_i} = (1,0)$ and $\deg L_{Q_j} = (0,1)$.

We can associate to $\X$ a tuple $\lambda = (\lambda_1,\ldots,\lambda_r)$
where $\lambda_i = \#\{P \times Q \in \X ~|~ P = P_i\}$.
After relabeling the points, we can assume that $\lambda_1 \geq \cdots \geq \lambda_r$.
Note that $\lambda$ is then a {\bf partition} of $|\X|$.
Associated to $\lambda$ is another partition  $\lambda^* =
(\lambda_1^*,\ldots,\lambda_{\lambda_1}^*)$, called the {\bf
conjugate} of $\lambda$, where $\lambda^*_i = \#\{\lambda_j \in
\lambda ~|~ \lambda_j \geq i\}$.  When $\X$ is ACM, we can relabel
the points so that $\lambda_j^* = \#\{P \times Q \in \X ~|~ Q = Q_j\}$
(this can be deduced from \cite[Theorem 4.8]{VT2}).
Thus, when $\X$ is an ACM set of points in $\popo$, by relabeling the points
and permuting the lines of degree $(1,0)$ and $(0,1)$, we can always assume
that $\X$ resembles the Ferrer's diagram of the partition $\lambda$.
As an example, the set of points
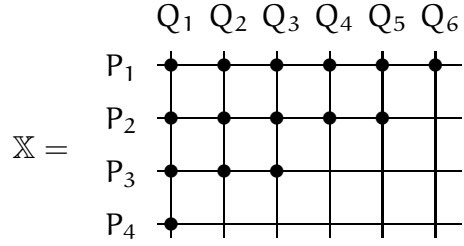
\begin{center}
\begin{picture}(150,100)(25,0)
\put(0,40){$\X = $} \put(60,10){\line(0,1){70}}
\put(80,10){\line(0,1){70}} \put(100,10){\line(0,1){70}}
\put(120,10){\line(0,1){70}} \put(140,10){\line(0,1){70}}
\put(160,10){\line(0,1){70}}

\put(54,90){$Q_1$} \put(74,90){$Q_2$} \put(94,90){$Q_3$}
\put(114,90){$Q_4$} \put(134,90){$Q_5$} \put(154,90){$Q_6$}

\put(55,15){\line(1,0){115}} \put(55,35){\line(1,0){115}}
\put(55,55){\line(1,0){115}} \put(55,75){\line(1,0){115}}

\put(35,11){$P_4$} \put(35,31){$P_3$} \put(35,51){$P_2$}
\put(35,71){$P_1$}

\put(60,15){\circle*{5}} \put(60,35){\circle*{5}}
\put(60,55){\circle*{5}} \put(60,75){\circle*{5}}

\put(80,55){\circle*{5}} \put(80,35){\circle*{5}}
 \put(80,75){\circle*{5}}

\put(100,35){\circle*{5}} \put(100,55){\circle*{5}}
 \put(100,75){\circle*{5}}

\put(120,75){\circle*{5}} \put(120,55){\circle*{5}}

\put(140,55){\circle*{5}} \put(140,75){\circle*{5}}

\put(160,75){\circle*{5}}
\end{picture}
\end{center}
is an ACM set of points corresponding to $\lambda = (6,5,3,1)$.
For this set of points $\lambda^* = (4,3,3,2,2,1)$;  the first three
in $\lambda^*$ corresponds
to the fact that there are three points which have second
coordinate $Q_2$.

When $\X$ is ACM, some of the algebraic invariants of $\Ix$ can be deduced from $\lambda$.

\begin{theorem}\label{acmprop}
Let $\X$ be an ACM set of points in $\popo$, and let $\lambda = (\lambda_1,\ldots,\lambda_r)$
be the associated partition.
\begin{enumerate}
\item[$(i)$] The minimal $\N^2$-graded resolution of $R/\Ix$ has form
\[0 \rightarrow \bigoplus_{(i,j) \in S_2} R(-i,-j) \rightarrow
\bigoplus_{(i,j) \in S_1} R(-i,-j) \rightarrow R \rightarrow R/\Ix \rightarrow 0\]
where
\begin{eqnarray*}
S_1 &:=& \{(r,0),(0,\lambda_1)\} \cup \{(i-1,\lambda_i) ~|~ \lambda_i - \lambda_{i-1} <0\}, ~\text{and}\\
S_2 &:=& \{(r,\lambda_r)\} \cup \{(i-1,\lambda_{i-1}) ~|~ \lambda_i - \lambda_{i-1} < 0 \}.
\end{eqnarray*}
\item[$(ii)$]  Assume that the points of $\X$
have been relabeled so that $\X$ resembles the Ferrer's diagram of $\lambda$.
Let $\{i_1,\ldots,i_l\} \subseteq \{1,\ldots,r\}$ be the locations of the
``drops'' in $\lambda$, that is,
\[\lambda_1 = \ldots = \lambda_{i_1-1} > \lambda_{i_1} = \ldots = \lambda_{i_2-1} > \lambda_{i_2} = 
\cdots \]
Then a minimal set of generators of $\Ix$ is given by
\[\{L_{P_1}\cdots L_{P_r},L_{Q_1}\cdots L_{Q_{\lambda_1}}\} \cup \{G_1,\ldots,G_l\}\]
where $G_k = L_{P_1}\cdots L_{P_{i_k-1}}L_{Q_1}\cdots L_{Q_{\lambda_{i_k}}}$ for $k = 1,\ldots,l$.
\end{enumerate}
\end{theorem}

\begin{proof}
The statement $(i)$ is \cite[Theorem 5.1]{VT2}.  For $(ii)$, because $\X$
has been relabeled to resemble a Ferrer's diagram, it is straightforward
to verify that each element of
$\{L_{P_1}\cdots L_{P_r},L_{Q_1}\cdots L_{Q_{\lambda_1}}\} \cup \{G_1,\ldots,G_l\}$
vanishes at all the points of $\X$ and thus belongs to $\Ix$.
To see that these are the minimal generators, it suffices to compare the degrees
of each element with the elements in the set $S_1$ from part $(i)$.
\end{proof}

A set
$\X \subseteq \popo$ satisfies {\bf property $(\star)$} if whenever $P_1 \times Q_1$ and
$P_2 \times Q_2$ are two points in $\X$ with $P_1 \neq P_2$ and $Q_1 \neq Q_2$,
then either $P_1 \times Q_2 \in \X$ or $P_2 \times Q_1 \in \X$ (or both)
are in $\X$.  We then have:

\begin{theorem}[{\cite[Theorem 4.3]{GVT1}}]\label{theorem*}
A finite set of points $\X$ in $\pr^1 \times \pr^1$
is ACM if and only if $\X$ satisfies property $(\star)$.
\end{theorem}

\subsection{Separators and resolutions in $\popo$}
We begin by describing how to compute $\deg_{\X}(P)$ for each point
$P \in \X \subseteq \popo$ when $\X$ is ACM.
A similar result was given by
Marino \cite[Proposition 7.4]{m1}, but using the language of
``left segments''.

\begin{lemma}\label{D}  $(i)$ Let $\{Q_1,\ldots,Q_b\}$ be $b\geq 2$
distinct points in $\pr^1$, and let $P_1$ be any point of $\pr^1$
(we allow the case that $P_1 = Q_i$ for some $i$).  Consider the
set of points
\[\X = \{P_1\times Q_1,P_1 \times Q_2,\ldots,P_1 \times Q_b\} \subseteq \popo.\]
Then $\X$ is ACM, and furthermore, $\deg_{\X}(P_1\times Q_1) =
\{(0,b-1)\}$.

\noindent $(ii)$  Let $\{P_1,\ldots,P_a\}$ be $a\geq 2$ distinct
points in $\pr^1$, and let $Q_1$ be any point of $\pr^1$ (we
allow the case that $Q_1 = P_i$ for some $i$).  Consider the set
of points
\[\X = \{P_1\times Q_1,P_2 \times Q_1,\ldots,P_a \times Q_1\} \subseteq \popo.\]
Then $\X$ is ACM, and furthermore, $\deg_{\X}(P_1\times Q_1) =
\{(a-1,0)\}$.
\end{lemma}
\begin{proof}
We only prove $(i)$ since the second statement is similar. By
Theorem \ref{theorem*}, we have that $\X$ is ACM, and so by
Theorem \ref{uniqueMultsep}, we have
$\deg_{\X}(P_1 \times Q_1) = \ua$ for some $\ua
\in \N^2$.  If $L_{Q_i}$ denotes the degree $(0,1)$ form that
passes through $Q_i$, then the form $L_{Q_2}L_{Q_3}\cdots
L_{Q_b}$ is a separator of degree $(0,b-1)$, so we must have
$(0,b-1) \succeq \ua$.

If $(0,b-1)\succ \ua$, then there exists a separator $F \neq 0$ with
$\deg F = \ua = (0,b')$ for some $b' < b-1$.  On the other hand,
the bigraded Hilbert function of $\Y = \X \setminus \{P_1 \times Q_1\}$ is
\[
H_{\Y} = \begin{bmatrix}
1 & 2 & 3 & \cdots & b-2 & b-1 & b- 1 & \cdots \\
1 & 2 & 3 & \cdots & b-2 & b-1 & b- 1 & \cdots \\
\vdots & \vdots& \vdots& & \vdots &\vdots & \vdots &\ddots
\end{bmatrix}  \]
(we write $H_{\Y}$ as an infinite matrix
where the $(i,j)$th entry of the matrix equals $H_{\Y}(i,j)$
where the indexing starts at $0$)
which implies that $(I_{\Y})_{i,j} = 0$ for all $(0,b-2) \succeq
(i,j)$. Since $F \in (I_{\Y})_{(0,b')}$ with $b' < b-1$ this means
$F = 0$, a contradiction. So, $\ua = (0,b-1)$, as desired.
\end{proof}

When $\X$ is an ACM set of points in $\popo$,  the degree of every point in $\X$ 
is found by simply counting the points which share the same  first and second
coordinate.

\begin{theorem}\label{degP1xP1}
Let $\X$ be an ACM set of points in $\popo$.  For any $P \times Q
\in \X$ let
\[\X_{P,1} = \{P \times Q, P \times Q_2,\ldots, P\times Q_b\} \subseteq \X\]
be all the points of $\X$ whose first coordinate is $P$, and let
\[\X_{Q,2} = \{P \times Q, P_2 \times Q,\ldots, P_a\times Q\} \subseteq \X\]
be all the points of $\X$ whose second coordinate is $Q$.  Then
\[\deg_{\X}(P\times Q) = \{(|\X_{Q,2}|-1,|\X_{P,1}|-1)\} = \{(a-1,b-1)\}.\]
\end{theorem}

\begin{proof}
Because $\X$ is ACM, $\deg_{\X}(P\times Q) = \alpha$ for some $\alpha \in \N^2$. Let
$L_{P_i}$ be the degree $(1,0)$ form that passes through $P_i$
for $i=2,\ldots,a$ and let $L_{Q_j}$ be the degree $(0,1)$ form
that passes through $Q_j$ for $j=2,\ldots,b$.  We will show that
$F = L_{P_2}\cdots L_{P_{a}}L_{Q_2}\cdots L_{Q_b}$ is a
minimal separator of $P \times Q$.

By construction, $F(P\times Q) \neq 0$.  Now consider any point
$P' \times Q' \in \X \setminus \{P \times Q\}$. If $P' \in
\{P_2,\ldots,P_a\}$ or $Q' \in \{Q_2,\ldots,Q_b\}$, then $F(P'
\times Q') = 0$.  So, suppose $P' \not\in \{P_2,\ldots,P_a\}$ and
$Q' \not\in \{Q_2,\ldots,Q_b\}$.  Now $\X$ satisfies property $(\star)$
by Theorem \ref{theorem*}.   So,
since $P' \times Q'$ and $P \times Q$ are in $\X$, then either $P'
\times Q \in \X$, in which case $P' \in \{P_2,\ldots,P_a\}$, a
contradiction, or $P \times Q' \in \X$, in which case $Q' \in
\{Q_2,\ldots,Q_b\}$, a contradiction.  Hence $F$ is a separator for
$P$ of degree $(a-1,b-1)$, whence $(a-1,b-1) \succeq \alpha$.

Suppose that $\alpha \prec (a-1,b-1)$.  So, there is a minimal
separator $F'\neq 0$ with $\deg F' = \alpha = (\alpha_1,\alpha_2)$ with
$\alpha_1 < a-1$ or $\alpha_2 < b-1$.  Suppose that $\alpha_1 <
a-1$.  Now $F'$ is also a separator for $P \times Q$ from
$\X_{Q,2}$.  By Lemma \ref{D}, $deg_{\X_{Q,2}}(P\times Q) =
(a-1,0)$.  So, we must have $(a-1,0) \preceq \deg F'$.  But
$\alpha_1 < a-1$, which gives a contradiction.  So, $\alpha_1
\geq a-1$.  A similar argument implies $\alpha_2 \geq b-1$, and
thus $\deg_{\X}(P \times Q) = (a-1,b-1)$.
\end{proof}

\begin{remark} Suppose $\lambda = (\lambda_1,\ldots,\lambda_r)$
is the partition associated to $\X$,
and $\X$ resembles the Ferrer's
diagram of $\lambda$.  If $P_i \times Q_j \in \X$,
then the conclusion of Theorem \ref{degP1xP1} is equivalent to
$\deg_\X (P_i \times Q_j) = (\lambda_j^*-1,\lambda_i-1)$.
\end{remark}

We now prove the converse of Theorem \ref{pnrsepfrombetti}
for ACM points in $\popo$:

\begin{theorem}\label{acm}  Let $\X$ be an ACM set of points in $\popo$,
and suppose that
\[ 0 \rightarrow \mathbb{F}_2 = \bigoplus_{(i,j) \in S_2} R(-i,-j)
\rightarrow \mathbb{F}_1 \rightarrow R \rightarrow R/\Ix
\rightarrow 0 \] is the  minimal $\N^2$-graded free resolution of
$R/I_\X$. Let $P \in \X$ be any point.  Then $\Y = \X \setminus
\{P\}$ is ACM if and only if $\deg_{\X}(P)+(1,1) \in S_2$.
\end{theorem}

\begin{proof} In light of Theorem \ref{pnrsepfrombetti} (i), it
suffices to prove the converse statement.   As noted above,
we can assume that $\X$ resembles a Ferrer's diagram of some
partition $\lambda$.  By Theorem \ref{acmprop}
the shifts in $\mathbb{F}_2$ are
$S_2 = \{(r,\lambda_r)\} \cup \{(i-1,\lambda_{i-1})
 ~|~ \lambda_i - \lambda_{i-1} < 0 \}.$
We consider two cases: (1) $\lambda =
(\lambda_1,\ldots,\lambda_1)$ and (2) $\lambda =
(\lambda_1,\ldots,\lambda_r)$ with $\lambda_1 > \lambda_r$.

In the first case, $S_2 = \{(r,\lambda_r)\} =
\{(r,\lambda_1)\}$.  Moreover, $\lambda =
(\lambda_1,\ldots,\lambda_1)$ if and only if $\X$ is a complete
intersection of type $(\lambda_1,r)$, that is, $\X$ is a
grid of $\lambda_1 \times r$ points.  By Lemma \ref{degP1xP1},
each point $P \in \X$ has $\deg_{\X}(P) = (r-1,\lambda_1-1)$.
So, $\deg_{\X}(P) + (1,1) \in S_2$. But because $\X$ is a
complete intersection, $\Y = \X \setminus \{P\}$ is ACM because
$\Y$ still satisfies property $(\star)$.

For the second case, suppose $\deg_\X (P_i \times Q_j) + (1,1)
= (\lambda_j^*,\lambda_i) \in S_2$.  So,
either $(\lambda_j^*,\lambda_i) = (r,\lambda_r)$,
or there exists an $i' > i$ such that
$(\lambda_j^*,\lambda_i) = (i'-1,\lambda_{i'-1})$.
For the second statement,
because  $\lambda_i \geq \lambda_{i+1} \geq \cdots$ there exists some
$i'> i$ such that $\lambda_i = \cdots = \lambda_{i'-1} >
\lambda_{i'}$.
Since $(i'-1,\lambda_{i'-1})$ is the only tuple in $S_2$ whose
second coordinate is $\lambda_i = \lambda_{i'-1}$, then
$(i'-1,\lambda_{i'-1})$ is the element in $S_2$ that is equal to
$\deg_{\X}(P_i \times Q_j)+(1,1)$ and $\lambda_j^* = i'-1$. So,
if $\lambda_j^* = r$,
\[\X_{Q_j,2} = \{P_1, \times Q_j, P_2 \times Q_j,\ldots, P_i \times Q_j, \ldots,
P_{r}\times Q_j\}\]
is the set of all points in $\X$ with second coordinate $Q_j$,
and if $\lambda_j^* = i'-1$, then
\[\X_{Q_j,2} = \{P_1, \times Q_j, P_2 \times Q_j,\ldots, P_i \times Q_j, \ldots,
P_{i'-1}\times Q_j\} \subseteq \X\] is the set of all the points
in $\X$ with second coordinate is $Q_j$.

Suppose, for a contradiction, that $\Y = \X \setminus\{P_i \times Q_j\}$ is not ACM.
Thus  $\Y$ does not satisfy property $(\star)$.
We do the case that $\lambda_i = \lambda_{i'-1}$ first.
Because we have only removed the point $P_i \times Q_j$, this
means that there exist points $P_i \times Q'$ and $P' \times Q_j$
in $\Y$ with $P' \times Q' \not\in \Y$ (and clearly $P_i \times
Q_j \not\in \Y$).  Because $\X$ has the shape of the Ferrer's
diagram, we can take $P'=P_{c}$ with $c > i$.  To see this,
note that the Ferrer's shape implies that if $P_i \times Q' \in \X$,
then so are all the points $P_k \times Q'$ with $k < i$.  Thus,
because
$P_i \times Q'$ is in $\X$, but $P' \times Q' \not\in \X$, we
must have $P' = P_c$ with $c > i$.

On the other hand, again from the Ferrer's shape we
must also have
$\lambda_i > \lambda_c$, because $P_i \times Q' \in \X$,
but $P_c \times Q' \not\in \X$.
 But $P_c \times Q_j \in \X_{Q_j,2}$, and
thus $i < c \leq i'-1$.  But then we have $\lambda_i = \cdots=\lambda_c = \cdots =
\lambda_{i'-1}$, and thus $\lambda_c = \lambda_i < \lambda_i$, a
contradiction.  So, $\Y$ must have property $(\star)$, and must be
ACM by Theorem \ref{theorem*}

In the case that $\lambda_i = \lambda_r$, a similar argument would
show that there exists a point $P_c \times Q' \in \X$ with $c > r$
and $\lambda_c < \lambda_r$.  But this is not
possible since $\lambda_r$ is the smallest entry of $\lambda$.
So, again $\Y$ must have property $(\star)$, and must be ACM.
\end{proof}

\begin{example}
We illustrate the above ideas with the following set of
points in $\popo$:

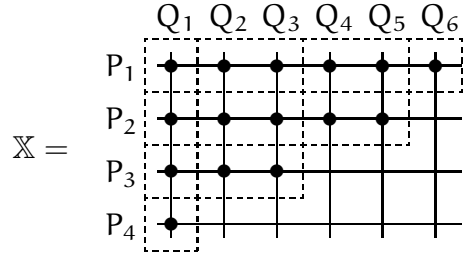
\begin{center}
\begin{picture}(150,100)(25,0)
\put(0,40){$\X = $} \put(60,10){\line(0,1){70}}
\put(80,10){\line(0,1){70}} \put(100,10){\line(0,1){70}}
\put(120,10){\line(0,1){70}} \put(140,10){\line(0,1){70}}
\put(160,10){\line(0,1){70}}

\put(54,90){$Q_1$} \put(74,90){$Q_2$} \put(94,90){$Q_3$}
\put(114,90){$Q_4$} \put(134,90){$Q_5$} \put(154,90){$Q_6$}

\put(50,5){\dashbox{2}(20,80){}} \put(70,25){\dashbox{2}(40,60){}}
\put(110,45){\dashbox{2}(40,40){}}
\put(150,65){\dashbox{2}(20,20){}}
\put(50,25){\dashbox{2}(60,20){}}
\put(50,65){\dashbox{2}(100,20){}} \put(55,15){\line(1,0){115}}
\put(55,35){\line(1,0){115}} \put(55,55){\line(1,0){115}}
\put(55,75){\line(1,0){115}}

\put(35,11){$P_4$} \put(35,31){$P_3$} \put(35,51){$P_2$}
\put(35,71){$P_1$}

\put(60,15){\circle*{5}} \put(60,35){\circle*{5}}
\put(60,55){\circle*{5}} \put(60,75){\circle*{5}}

\put(80,55){\circle*{5}} \put(80,35){\circle*{5}}
 \put(80,75){\circle*{5}}

\put(100,35){\circle*{5}} \put(100,55){\circle*{5}}
 \put(100,75){\circle*{5}}

\put(120,75){\circle*{5}} \put(120,55){\circle*{5}}

\put(140,55){\circle*{5}} \put(140,75){\circle*{5}}

\put(160,75){\circle*{5}}
\end{picture}
\end{center}

\noindent The associated partition is $\lambda = (6,5,3,1)$ and
$\lambda^* = (4,3,3,2,2,1)$.

We have divided the set of points into a series of boxes (the
dashed boxes).  Every point in the same box has the same degree.
For example $P_1 \times Q_2$ and $P_1 \times Q_3$ both have
degree $(2,5)$.  For this set of points, the shifts that appear
at the end of the minimal resolution of $R/\Ix$ are:
\[S_2 = \{(4,1),(1,6),(2,5),(3,3)\}.\]
The points in $\X$ in the ``outside'' boxes, i.e., the box
containing $P_4 \times Q_1$, the box containing $P_3 \times Q_2$
and $P_3 \times Q_3$, the box containing $P_2 \times Q_4$ and
$P_2 \times Q_5$, and the box containing $P_1 \times Q_6$, all
have the property that $\deg_{\X}(P_i \times Q_j) +(1,1) \in
S_2$.  For example, $\deg_\X(P_3 \times Q_3) + (1,1) = (3,3)$.
If we remove any point from these boxes, the resulting set of
points will still be ACM.  On the other hand, if we remove any
point from an ``inside'' box, the resulting set of points will
not be ACM.  For example, if $P_2 \times Q_3$ is removed, then $\X
\setminus \{P_2 \times Q_3\}$ no longer satisfies property
$(\star)$, and thus is not ACM.
\end{example}

Let $\nu(I)$ denote the minimal number of generators of a multihomogeneous ideal $I$.

\begin{lemma}\label{numgens}
Suppose that $\X \subseteq \popo$ is ACM, but $\Y = \X \setminus \{P\}$
is not ACM for some $P \in \X$.  Then $\nu(I_{\Y}) = \nu(\Ix)+1$.
\end{lemma}

\begin{proof}
Because $\X$ is ACM, $|\deg_{\X}(P)| = 1$.  Let $F$ be a minimal
separator of $\deg_{\X}(P)$.  By Theorem \ref{teopt} we have $I_{\Y} = (\Ix,F)$,
and hence $\nu(I_{\Y}) \leq \nu(\Ix) +1$.

Let $\lambda = (\lambda_1,\ldots,\lambda_r)$ be the partition associated to $\X$,
and relabel $\X$ so $\X$ resembles the Ferrer's diagram of $\lambda$.
Note that $\lambda_1 > \lambda_r$, because if $\lambda_1 = \lambda_r$, then
$\X$ would be a complete intersection, in which case $\X \setminus \{P\}$ is ACM
for all $P \in \X$.

Assume that $P = P_i \times Q_j$.  Because $\Y = \X \setminus \{P_i \times Q_j\}$
is not ACM, the set $\Y$ does not satisfy $(\star)$.  In particular,
the points $P_i \times Q_{\lambda_i}$ with $\lambda_i > j$
and $P_{\lambda_j^*} \times Q_j$ with $\lambda_j^* > i$ are in both $\Y$,
but neither $P_i \times Q_j$ or
$P_{\lambda_j^*} \times Q_{\lambda_i}$ are in $\Y$.  Note that this also
implies that $P_{\lambda_j^*} \times Q_{\lambda_i} \not\in \X$.  By
Theorem \ref{degP1xP1}, we have $\deg_{\X}(P_i\times Q_j) = (\lambda_j^*-1,\lambda_i-1)$;
in particular, a minimal separator of this point is
\[F = L_{P_1}\cdots\hat{L}_{P_i}\cdots L_{P_{\lambda_j^*}}L_{Q_1}\cdots \hat{L}_{Q_j}
\cdots L_{Q_{\lambda_i}}\]
where $~~\hat{\space}~~$ denotes the term is omitted.

By Theorem \ref{acmprop}, if  $\{i_1,\ldots,i_l\} \subseteq \{1,\ldots,r\}$ are the locations of the
``drops'' in $\lambda$, then the minimal generators of $\Ix$ are
\[\mathcal{M} = \{L_{P_1}\cdots L_{P_r},L_{Q_1}\cdots L_{Q_{\lambda_1}}\} \cup \{G_1,\ldots,G_l\}\]
where $G_k = L_{P_1}\cdots L_{P_{i_k-1}}L_{Q_1}\cdots L_{Q_{\lambda_{i_k}}}$ for $k = 1,\ldots,l$.
If $\nu(I_{\Y}) < \nu(\Ix)+1$, then because $F \not\in\Ix$, there exists a minimal generator
$G$ such that
\[G = HF + \sum_{F_i \in \mathcal{M}\setminus\{G\}} H_iF_i.\]
Now, by degree considerations, if $H_i \neq 0$, then we must have $\deg G \succeq \deg F_i$.
But by Theorem \ref{acmprop} (i), for any two minimal generators $F_i,F_j$ of $\Ix$,
we have $\deg F_i \not\succeq \deg F_j$ and $\deg F_j \not\succeq \deg F_i$.
Thus, if $\nu(I_{\Y}) < \nu(\Ix)+1$, in the sum above we have $H_i = 0$ for all $i$, and
thus there must be a generator $G$ such that $G = HF$, and hence
$\deg G \succeq \deg F$.

Again, by degree considerations, since $\deg F \succeq (1,1)$,
$G \neq L_{P_1}\cdots L_{P_r}$ or $L_{Q_1}\cdots L_{Q_{\lambda_1}}$.
So, consider $i_k \in \{i_1,\ldots,i_l\}$.
If $i_k \leq i$, then $\deg G_k = (i_k-1,\lambda_{i_k})$.  But then $i_k \leq i < \lambda_j^*$,
and thus $\deg G_k \not\succeq (\lambda_j^*,\lambda_i) = \deg F$.
On the other hand if $i_k > \lambda_j^*$, then
$\lambda_{i_k} < j < \lambda_i$.  But
then
$\deg G_k = (i_k-1,\lambda_{i_k}) \not\succeq (\lambda_j^*-1,\lambda_{i}-1) = \deg F$
since $\lambda_{i_k} < \lambda_i-1$.

So, it remains to consider the case that $i_k \in \{i_1,\ldots,i_l\}$
and $i< i_k \leq \lambda_j^*$.  We then have
\[ G_k = L_{P_1}\cdots L_{P_{i_k-1}}L_{Q_1}\ldots,L_{Q_{\lambda_{i_k}}} =
H L_{P_1}\cdots\hat{L}_{P_i}\cdots L_{P_{\lambda_j^*}}L_{Q_1}\cdots \hat{L}_{Q_j}
\cdots L_{Q_{\lambda_i}} = HF.\]
But since $i < i_k $, we have $\lambda_i \geq \lambda_{i_k-1} > \lambda_{i_k}$.
Thus $F$ cannot divide $G_k$ since $L_{Q_{\lambda_i}}$ divides $F$, but not $G_k$.
We have thus shown that for every generator of $\Ix$, $F$ cannot
divide it,  thus providing the desired contradiction.
\end{proof}

By Theorem \ref{pnrsepfrombetti} (ii), if $P \in \X \subseteq \pnr$ is chosen
so that $\Y = \X \setminus \{P\}$ is not ACM, then
$\depth(R/I_{\Y}) = r-1$.  When $\X \subseteq \popo$, we can
prove a stronger result.

\begin{theorem}\label{nonacm}
 Let $\X$ be an ACM set of points in $\popo$,
and suppose that
\[ 0 \rightarrow \mathbb{F}_2
\rightarrow \mathbb{F}_1 \rightarrow R \rightarrow R/\Ix
\rightarrow 0 \] is the  minimal $\N^2$-graded free resolution of
$R/I_\X$. Let $P \in \X$ be a point with $\deg_{\X}(P) = (\alpha_1,\alpha_2)$.
Then $\Y=\X \setminus \{P\}$ is not ACM if and only
if it has a minimal $\N^2$-graded free resolution of type
\begin{equation}\label{resY} 0 \rightarrow R(-\alpha_1-1,-\alpha_2-1)
 \rightarrow
\begin{array}{c}
\mathbb{F}_2 \\
\oplus \\
R(-\alpha_1-1,-\alpha_2) \\
\oplus \\
R(-\alpha_1,-\alpha_2-1)
\end{array}
\longrightarrow
\begin{array}{c}
\mathbb{F}_1 \\
\oplus \\
R(-\alpha_{1},-\alpha_{2})
\end{array}
\rightarrow R\rightarrow R/I_{\Y}\rightarrow
0
\end{equation}
\end{theorem}

\begin{proof}
If $\Y$ is a set of points in $\popo$ with a minimal bigraded
free resolution of type (\ref{resY}), since it has length $3$,
then $\Y$ is not ACM.
So, suppose that $\Y = \X \setminus \{P\}$ is not ACM.
As shown in the proof of Theorem \ref{pnrsepfrombetti}, $R/I_{\Y}$
has resolution of type (\ref{resY}).  It suffices to
show that the resolution is not minimal.

We first note that the resolution cannot be shortened since
$\depth(R/I_{\Y}) < 2$.  Thus, if the resolution of (\ref{resY})
were not minimal, some shift in $\mathbb{F}_1 \oplus R(-\alpha_1,-\alpha_2)$
would have to cancel out with some shift in
$\mathbb{F}_2 \oplus R(-\alpha_1-1,-\alpha_2)
\oplus
R(-\alpha_1,\alpha_2-1)$.  But if there were such a cancellation,
that would imply that $\nu(I_{\Y}) \leq \nu(\Ix)$, contradicting
Lemma \ref{numgens}.  Thus $R/I_{\Y}$ has a minimal resolution
of type (\ref{resY}).
\end{proof}

\begin{example}
Suppose we know $\X \subseteq \pnr$ is not ACM, and in fact, we
know the $\N^r$-graded resolution.  It is tempting to speculate
that the rank of the last syzygy module gives us information
about the minimal number of points one should add to $\X$ to make
the set of points ACM.  Unfortunately, no clear correspondence is
known.   For example, let $P_i\in \pr^1$ for $i=1,\ldots,5$ be
distinct points and let $\X$ be the following set of points of
type $P_{i,j}=P_i\times P_j$ in $\popo$:
\[\X=\{P_{1,1},P_{1,3},P_{1,5},P_{2,2},P_{2,4},P_{2,5},P_{3,1},
P_{3,2},P_{3,3},P_{4,1},P_{4,4}\}.\] Then, using {\tt CoCoA}, a
resolution of $R/I_{\X}$ has the form
\[0\rightarrow R^4\rightarrow R^{10}\rightarrow R^7\rightarrow
R\rightarrow R/I_{\X}\rightarrow 0\] where we have suppressed all
the bigraded shifts. The rank of the last syzygy module is $4$.
However, to make $\X$ ACM, we need to add at least 5 points:
$P_{1,2},P_{1,4},P_{2,1},P_{2,3}$, and $P_{3,4}$.
\end{example}



\begin{thebibliography}{99}

\bibitem{abm} S. Abrescia, L. Bazzotti, L. Marino,
Conductor degree and socle degree.  Matematiche (Catania)  {\bf
56}  (2001) 129--148 (2003).


\bibitem{b} L. Bazzotti, M.  Casanellas,
Separators of points on algebraic surfaces. J. Pure Appl. Algebra
{\bf 207} (2006) 319--326.

\bibitem{b2} L. Bazzotti,
Sets of points and their conductor. J. Algebra {\bf 283} (2005)
799--820.

\bibitem{Co}
CoCoATeam, CoCoA: a system for doing Computations in Commutative
Algebra. Available at {\tt http://cocoa.dima.unige.it}


\bibitem{E} D. Eisenbud, {\it Commutative algebra with a view toward algebraic geometry.} 
Graduate Texts in Mathematics, {\bf 150}. Springer-Verlag, New York, 1995.

\bibitem{gmr} A.V. Geramita, P. Maroscia, L. Roberts,
The Hilbert function of a reduced $k$-algebra. J. Lond. Math.
Soc. (2) {\bf 28} (1983) 443-452.

\bibitem{GuMaRa} S. Giuffrida, R. Maggioni, A. Ragusa,
On the postulation of $0$-dimensional subschemes on a smooth
quadric. Pacific J. Math. {\bf 155} (1992) 251--282.

\bibitem{GuMaRa2} S. Giuffrida, R. Maggioni, A. Ragusa,
Resolutions of $0$-dimensional subschemes of a smooth quadric. In
{\it  Zero-dimensional schemes (Ravello, 1992)}, de Gruyter,
Berlin (1994) 191--204.


\bibitem{GuMaRa3} S. Giuffrida, R. Maggioni, A. Ragusa,
Resolutions of generic points lying on a smooth quadric.
Manuscripta Math. {\bf 91} (1996) 421--444.


\bibitem {Gu2} E. Guardo, Fat point schemes on a smooth quadric.
J Pure Appl. Algebra \textbf {162} (2001) 183-208.

\bibitem{GVT} E. Guardo, A. Van Tuyl,
Fat points of $\popo$ and their Hilbert Functions.
Canad. J. Math {\bf 56} (2004) 716-741.

\bibitem{GVT1} E. Guardo, A. Van Tuyl, ACM sets
of points in multiprojective space.  (2007) To appear in
Collect. Math. {\tt arXiv:0707.3138v2}

\bibitem{guida}
M. Guida,
Conductor degree of points with maximal Hilbert function and socle degree.
Ricerche Mat. {\bf 52} (2003) 145--156.


\bibitem{HO} T. Harima, H. Okuyama,
The conductor of some special points in $P\sp 2$.
J. Math. Tokushima Univ.  {\bf 28}  (1994) 5--18.


\bibitem{m1} L. Marino,
Conductor and separating degrees for sets of points in $\pr^r$ and in $\popo$.   Boll. Unione Mat.
Ital. Sez. B Artic. Ric. Mat. (8) {\bf 9} (2006) 397--421.

\bibitem{m2} L. Marino,
The minimum degree of a surface that passes through all the
points of a 0-dimensional scheme but a point $P$.  {\it Algebraic
structures and their representations}, 315--332, Contemp. Math.,
{\bf 376}, Amer. Math. Soc., Providence, RI, 2005.

\bibitem{m3} L. Marino, A characterization of ACM 0-dimensional subschemes
of $\popo$. In preparation.

\bibitem{O}  F. Orecchia,
Points in generic position and conductors of curves with ordinary
singularities. J. London Math. Soc. (2) {\bf 24} (1981) 85--96.

\bibitem{S} A. Sodhi,
The conductor of points having the Hilbert function of a complete intersection in $P\sp 2$.
Canad. J. Math. {\bf 44} (1992) 167--179.

\bibitem{SVT} J. Sidman, A. Van Tuyl,
Multigraded regularity: syzygies and fat points.  Beitr\"age
Algebra Geom. {\bf 47}  (2006) 67--87.


\bibitem{VT1} A. Van Tuyl, The border of the Hilbert function of a set of points in
$\pnk$.  J. Pure
Appl. Algebra  {\bf 176} (2002) 223--247.

\bibitem{VT2} A. Van Tuyl, The Hilbert functions of ACM sets of points in
$\pnk$.  J.
Algebra  {\bf 264} (2003) 420--441.

\end{thebibliography}
\end{document}